\documentclass{elsarticle}

\usepackage{hyperref}
\usepackage{microtype}
\usepackage{amsmath}
\usepackage{amssymb,amsfonts,amsthm,latexsym,mathrsfs}
\usepackage{eucal}
\usepackage{color}

\theoremstyle{plain}
\newtheorem{theorem}[subsection]{{\bf Theorem}}
\newtheorem*{theorem*}{{\bf Theorem}}
\newtheorem{corollary}[subsection]{{\bf Corollary}}
\newtheorem*{corollary*}{{\bf Corollary}}
\newtheorem{proposition}[subsection]{{\bf Proposition}}
\newtheorem{lemma}[subsection]{{\bf Lemma}}

\theoremstyle{definition}

\theoremstyle{remark}

\numberwithin{equation}{subsection}

\def\Sym{\mathrm{Sym}}

\newcommand{\PSL}{\mathrm{PSL}}
\newcommand{\SL}{\mathrm{SL}}
\newcommand{\PGL}{\mathrm{PGL}}

\newcommand{\GF}{\mathrm{GF}}

\newcommand{\GL}{\mathrm{GL}}

\def \udot {{}^{\textstyle .}}

\def\X{\mathfrak X}

\newcommand{\Q}{\mathrm{Q}}

\newcommand{\Aut}{\mathrm{Aut}}

\newcommand{\Inn}{\mathrm{Inn}}
\newcommand{\Syl}{\mathrm{Syl}}

\newcommand{\Alt}{\mathrm{Alt}}
\newcommand{\Dih}{\mathrm{Dih}}

\def \Z {\mathbb Z}

\def \Aut{ \mathrm {Aut}}

\def \Mat{\mbox {\rm Mat}}

\def \GL{\mathrm {GL}}

\journal{Journal of Pure and Applied Algebra}

\begin{document}

\begin{frontmatter}

\title{Locally finite groups in which every non-cyclic subgroup is self-centralizing}

\author[1]{Costantino Delizia\corref{mycorrespondingauthor}}
\cortext[mycorrespondingauthor]{Corresponding author}
\ead{cdelizia@unisa.it}

\author[2]{Urban Jezernik}
\ead{urban.jezernik@imfm.si}

\author[3]{Primo\v z Moravec}
\ead{primoz.moravec@fmf.uni-lj.si}

\author[1]{Chiara Nicotera}
\ead{cnicoter@unisa.it}

\author[4]{Chris Parker}
\ead{c.w.parker@bham.ac.uk}

\address[1]{University of Salerno, Italy}
\address[2]{Institute of Mathematics, Physics, and Mechanics, Ljubljana, Slovenia}
\address[3]{University of Ljubljana, Slovenia}
\address[4]{University of Birmingham, United Kingdom}

\begin{abstract}
Locally finite groups having the property that every non-cyclic subgroup contains its centralizer are completely classified.
\end{abstract}


\end{frontmatter}


\section{Introduction}
\label{s:intro}

A subgroup $H$ of a group $G$ is {\it self-centralizing} if the centralizer $C_G(H)$ is contained in $H$. In \cite{Del13} it has been remarked that a locally graded group in which all non-trivial subgroups are self-centralizing has to be finite; therefore it has to be either cyclic of prime order or non-abelian of order being the product of two different primes.

In this article, we consider the more extensive class $\X$ of all groups in which every non-cyclic subgroup is self-centralizing. In what follows we   use the term $\X$-groups in order to denote groups in the class $\X$. The study of properties of $\X$-groups was initiated in \cite{Del13}. In particular, the first four authors determined the structure of finite $\X$-groups which are either nilpotent, supersoluble or simple.

In this paper, Theorem~\ref{one} gives a complete classification of finite $\X$-groups. We remark that this result does not depend on classification of the finite simple groups rather only on the classification of groups with dihedral or semidihedral Sylow $2$-subgroups.
We  also determine the infinite soluble $\X$-groups, and the infinite locally finite $\X$-groups, the results being presented in  Theorems~\ref{t:soluble} and \ref{t:locallyfinite}. It turns out that these latter groups are suitable finite extensions either of the infinite cyclic group or of a Pr\"ufer $p$-group, $\Z_{p^\infty}$, for some prime $p$. Theorem~\ref{t:locallyfinite} together with Theorem~\ref{one} provides a complete classification of locally finite $\X$-groups.

We follow \cite{Asc93} for basic group theoretical notation. In particular, we note that $F^*(G)$ denotes the generalized Fitting subgroup of $G$, that is the subgroup of $G$ generated by all subnormal nilpotent or quasisimple subgroups of $G$. The latter subgroups are the components of $G$. We see from  \cite[Section 31]{Asc93} that distinct components commute.
The fundamental property of the generalized Fitting subgroup  that we shall use is that it contains its centralizer in $G$ \cite[(31.13)]{Asc93}.  We denote the alternating group and symmetric group of degree $n$ by $\Alt(n)$ and $\Sym(n)$ respectively. We use standard notation for the classical groups. The notation $\Dih(n)$ denotes the dihedral group of order $n$ and $\Q_8$ is the quaternion group of order $8$. The term quaternion group will cover groups  which  are often called generalized quaternion groups. The cyclic group of order $n$ is represented simply by $n$, so for example $\Dih(12) \cong 2 \times \Dih(6)\cong 2 \times \Sym(3)$.
Finally $\Mat(10)$ denotes the Mathieu group of degree 10. The  Atlas \cite{Con85} conventions are  used for group extensions. Thus, for example, $p^2{:}\SL_2(p)$ denotes the split extension of an elementary abelian group of order $p^2$ by $\SL_2(p)$.

\section*{ Acknowledgment}
We would like to thank Professor Hermann Heineken  for pointing out an oversight in the statement of (2.1.3) of Theorem~\ref{one}. We also thank the referee for valuable suggestions.
\section{Finite $\X$-groups}
\label{s:finite}

In this section we determine all the finite groups belonging to the class $\X$. The main result is the following.
\begin{theorem}\label{one}
Let $G$ be a finite $\X$-group.  Then one of the following holds:
\begin{enumerate}
\item[$(1)$] If $G$ is nilpotent, then either
\begin{enumerate}
\item[$(1.1)$] $G$ is cyclic;
\item[$(1.2)$] $G$ is elementary abelian of order $p^2$ for some prime $p$;
\item[$(1.3)$]  $G$ is an extraspecial $p$-group of order $p^3$ for some odd prime $p$; or
\item[$(1.4)$] $G$ is a dihedral, semidihedral or quaternion $2$-group.
\end{enumerate}
\item[$(2)$] If $G$ is supersoluble but not nilpotent, then, letting $p$ denote the largest prime divisor of $|G|$ and $P \in \Syl_p(G)$, we have that $P$ is a  normal subgroup of $G$ and  one of the following holds:
\begin{enumerate}
\item[$(2.1)$] $P$ is cyclic and either
\begin{enumerate}
\item[$(2.1.1)$] $G \cong  D \ltimes C$, where $C$ is cyclic,  $D$ is cyclic and every non-trivial element of $D$ acts fixed point
freely on $C$ (so $G$ is a Frobenius group);
\item[$(2.1.2)$]  $G = D\ltimes C$, where $C$ is a cyclic group of odd order, $D$ is a
quaternion group, and $C_G(C) = C\times D_0$ where $D_0$ is a cyclic subgroup of index $2$ in $D$ with $G/D_0$  a dihedral group; or
\item[$(2.1.3)$] $ G = D \ltimes C$, where $D$ is a cyclic $q$-group, $C$ is a cyclic $q'$-group
(here $q$ denotes the smallest prime dividing the order of $G$),  $1 < Z(G) < D$ and $G/Z(G)$ is a Frobenius group;
\end{enumerate}
\item[$(2.2)$] $P$ is extraspecial and $G$ is a Frobenius group with cyclic Frobenius complement of odd order dividing $p-1$.
\end{enumerate}
\item[$(3)$] If $G$ is not supersoluble and $F^*(G)$ is nilpotent, then either $(3.1)$ or $(3.2)$ below holds.
\begin{enumerate}
\item[$(3.1)$]  $F^*(G)$ is elementary abelian  of order $p^2$, $F^*(G)$ is a minimal normal subgroup of $G$ and one of the following holds:
\begin{enumerate}
\item[$(3.1.1)$] $p=2$ and $G \cong \Sym(4)$ or $G \cong \Alt(4)$;  or
\item[$(3.1.2)$] $p$ is odd and
$G= G_0\ltimes N$ is a Frobenius group  with Frobenius kernel $N$ and Frobenius complement $G_0$ which is itself  an $\X$-group. Furthermore, either
\begin{itemize}
\item[$(3.1.2.1)$] $G_0$ is cyclic of order dividing $p^2-1$ but not dividing $p-1$;
\item[$(3.1.2.2)$] $G_0$ is quaternion;
\item[$(3.1.2.3)$] $G_0$ is supersoluble as in $(2.1.2)$ with $|C|$ dividing $p-\epsilon $  where  $p\equiv \epsilon \pmod 4$;
\item[$(3.1.2.4)$] $G_0$ is supersoluble as in $(2.1.3)$ with $D$ a $2$-group, $C_D(C)$ a non-trivial maximal subgroup of $D$ and $ |C|$ odd dividing $p-1$ or $p+1$;
\item[$(3.1.2.5)$] $G_0 \cong \SL_2(3)$;
\item[$(3.1.2.6)$]  $G_0\cong \SL_2(3)\udot 2$ and  $p \equiv \pm 1 \pmod 8$; or
\item[$(3.1.2.7)$] $G_0 \cong \SL_2(5) $ and $60$ divides $p^2-1$.
\end{itemize}
\end{enumerate}
\item[$(3.2)$] $F^*(G)$ is extraspecial of order $p^3$ and one of the following holds:
\begin{enumerate}
\item[$(3.2.1)$] $G \cong \SL_2(3)$ or $G \cong \SL_2(3)\udot 2$ (with quaternion Sylow $2$-subgroups of order $16$); or
\item[$(3.2.2)$] $G= K\ltimes N$ where $N$ is extraspecial of order $p^3$  and exponent $p$  with $p$ an odd prime, $K$ centralizes $Z(N)$ and is cyclic of odd order dividing $p+1$. Furthermore, $G/Z(N)$ is a Frobenius group.
\end{enumerate}
\end{enumerate}
\item[$(4)$] If $F^*(G)$ is not nilpotent, then either
\begin{enumerate}
\item[$(4.1)$] $F^*(G) \cong \SL_2(p)$ where $p$ is a Fermat prime, $|G/F^*(G)|\le 2$ and $G$ has quaternion Sylow $2$-subgroups; or
\item[$(4.2)$] $G \cong \PSL_2(9)$, $\Mat(10)$ or $\PSL_2(p)$ where $p$ is a Fermat or Mersenne prime.
\end{enumerate}
\end{enumerate}
Furthermore, all the groups listed above are $\X$-groups.
\end{theorem}

We make a brief remark about the group $\SL_2(3)\udot 2$ and the groups appearing in part $(4.1)$ of Theorem~\ref{one} in the case $G > F^*(G)$. To  obtain such groups, take $F = \SL_2(p^2)$, then the groups in question are isomorphic to the normalizer in $F$ of the subgroup isomorphic to $\SL_2(p)$.  We denote these groups by $\SL_2(p)\udot 2$ to indicate that the extension is not split (there are no elements of order $2$ in the outer half of the group).

We shall repeatedly use the fact that if $L$ is a subgroup of an $\X$-group $X$, then $L$ is an $\X$-group. Indeed, if $H \le L$ is non-cyclic, then $C_L(H) \le C_X(H) \le H$.

The following elementary facts will facilitate our proof that the examples listed are indeed $\X$-groups.

 \begin{lemma}\label{l:one}
  The finite group $X$ is an $\X$-group if and only if $C_X(x)$ is an $\X$-group for all  $x \in X$ of prime order.
 \end{lemma}

\begin{proof} If $X$ is an $\X$-group, then, as $\X$ is subgroup closed, $C_X(x)$ is an $\X$-group for all $x \in X$ of prime order.  Conversely, assume that   $C_X(x)$ is an $\X$-group for all $x \in X$ of prime order (and hence of any order). Let $H \le X$ be non-cyclic. We shall show $C_X(H) \le H$.  If $C_X(H)=1$, then $C_X(H) \le H$ and we are done.  So assume  $x \in C_X(H)$ and $x \ne 1$. Then $H \le C_X(x)$ which is an $\X$-group. Hence $x \in C_{C_X(x)}(H) \le H$.  Therefore $C_X(H) \le H$, and   $X$ is an $\X$-group.
\end{proof}

\begin{lemma}\label{l:two} Suppose that $X$ is a Frobenius group with kernel $K$ and complement $L$. If $K$ and $L$ are $\X$-groups, then $X$ is an $\X$-group.
\end{lemma}

\begin{proof} Let $x \in X$ have prime order.  Then, as $K$ and $L$ have coprime orders, $x \in K$ or $x$ is conjugate to an element of $L$. But then, since $X$ is a Frobenius group,  either $C_X(x) \le K$ or $C_X(x)$ is conjugate to a subgroup of $L$. Since $K$ and $L$ are $\X$-groups, $C_X(x)$ is an $\X$-group. Hence $X$ is an $\X$-group by Lemma~\ref{l:one}.\end{proof}

The rest of this section is dedicated to the proof of Theorem~\ref{one}; therefore   $G$ always denotes a finite $\X$-group.  Parts (1) and (2) of  Theorem~\ref{one} are already proved in \cite[Theorems 2.2, 2.4, 3.2 and 3.4]{Del13}. However, our statement in (2.1.3) adds further detail which we now explain. So, for a moment, assume that $G$ is supersoluble, $q$ is the smallest prime dividing $|G|$, $D $ is a cyclic $q$-group and $C$ is a cyclic $q'$-group.  In addition,  $1\ne Z(G)= C_D(C)$.  Assume that   $d \in D \setminus Z(G)$. Then, as $d \not \in Z(G)$, $C$ is not centralized by $d$. By coprime action,  $C= [C,d] \times C_C(d)$ and so $Y=[C,d]\langle d \rangle $ is centralized by $C_C(d)$.  As $Y $ is non-abelian and $C_C(d) \cap Y= 1$, we deduce that $C_C(d)=1$.  Hence $G/Z(G)$ is a Frobenius group.
This means that  we can assume that (1) and (2) hold and, in particular, we  assume that $G$ is not supersoluble.

The following lemma provides the basic case subdivision of our proof.

\begin{lemma}\label{basic1} One of the following holds:
\begin{enumerate}
\item $F^*(G)$ is elementary abelian of order $p^2$ for some prime $p$.
\item $F^*(G)$ is extraspecial of order $p^3$ for some prime $p$.
\item $F^*(G)$ is quasisimple.
\end{enumerate}
\end{lemma}

\begin{proof}  Suppose first that $F^*(G)$ is nilpotent. Then its structure is given in part (1) of Theorem~\ref{one}. Suppose that $F^*(G)$ is cyclic.  Since $C_G(F^*(G))= F^*(G)$, we have $G/F^*(G)$ is isomorphic to a subgroup of $\Aut(F^*(G))$.
Because the automorphism group of a cyclic group is abelian, we have that $G$ is supersoluble. Therefore, by our assumption concerning $G$, $F^*(G)$ is not cyclic. Hence $F^*(G)$ is either elementary abelian of order $p^2$ for some prime $p$, is extraspecial of order $p^3$ for some odd prime $p$ or $F^*(G)$ is a dihedral, semidihedral or quaternion $2$-group. Since the automorphism groups of dihedral, semidihedral and quaternion groups of order at least $16$ are $2$-groups, we deduce that when  $p=2$ and $F^*(G)$ is non-abelian,  $F^*(G)$ is  extraspecial. This proves the lemma when $F^*(G)$ is nilpotent.

If $F^*(G)$ is not nilpotent, then there exists a component $K \le F^*(G)$.  As $F^*(G) = C_{F^*(G)}(K)K$ and $K$ is non-abelian, we have $F^*(G)= K$ and this is case (iii).
\end{proof}

\begin{lemma}\label{es}
Suppose that $p$ is a prime and $F^*(G)$ is extraspecial of order $p^3$. Then one of the following holds:
\begin{enumerate}
  \item $G \cong \SL_2(3)$,  $G \cong \SL_2(3)\udot 2$ (with quaternion Sylow $2$-subgroups of order $16$); or \item $G= NK$ where $N$ is extraspecial of order $p^3$ of exponent $p$ with $p$ an odd prime, $K$ centralizes $Z(N)$ and is cyclic of odd order dividing $p+1$. Furthermore, $G/Z(N)$ is a Frobenius group.
\end{enumerate}\end{lemma}

\begin{proof}   Let $N= F^*(G)$. We have that $N$ is extraspecial of order $p^3$ by assumption.   Suppose first that $p=2$, then we have $N \cong \Q_8$ as the dihedral group of order $8$ has no odd order automorphisms and $G$ is not a $2$-group.  Since $\Aut(\Q_8) \cong \Sym(4)$, $G/Z(N)$ is isomorphic to a subgroup of $\Sym(4)$ containing $\Alt(4)$.  If $G/Z(N) \cong \Alt(4)$, then $G= NT \cong \SL_2(3)$  where $T$ is a cyclic subgroup of order $3$. When $G/Z(N) \cong \Sym(4)$, taking $T \in \Syl_3(G)$, we have $NT \cong \SL_2(3)$, $N_G(T) $ has order $12$ and $N_G(T)/Z(N) \cong \Sym(3)$.  Since $N_G(T)$ is an $\X$-group and $N_G(T)$ is supersoluble, we see that $N_G(T)$ is a product $DT$ where $D$ is cyclic of order $4$ by (2.1.3).  Because the Sylow $2$-subgroups of $G$ are either dihedral, semidihedral or quaternion and $D\not \le N$, we see that $ND$ is quaternion. Thus $G \cong \SL_2(3)\udot 2$ as claimed in (i).

Assume that $p$ is odd.  We know that the outer automorphism group of $N$ is isomorphic to a subgroup of $ \GL_2(p)$ and $C_{\Aut(N)}(Z(N))/\Inn(N) $ is isomorphic to a subgroup of $ \SL_2(p)$.  Since $p$ is odd and the Sylow $p$-subgroups of $G$ are $\X$-groups, we have  $N \in \Syl_p(G)$ and  $G/N$ is a $p'$-group by part (1) of Theorem~\ref{one}. Set $Z= Z(N)$.
Since $G/N$ and $N$ have coprime orders, the Schur Zassenhaus Theorem says that $G$ contains a complement $K$ to $N$.
Set $K_1= C_K(Z)$. Then $K_1$ commutes with $Z$ and so $K_1$ is cyclic. If $K_1=1$, then $|K|$ divides $p-1$ and we find that $G$ is supersoluble, which is a contradiction. Hence $K_1 \ne 1$.
Let $x \in K_1$. Then $[N,x]$ and $ C_N(x)$ commute by the Three Subgroups Lemma. Hence $C_N(x)$ centralizes $[N,x]\langle x\rangle$ which is non-abelian. It follows that $[N,x]= N$ and $C_N(x)  =Z $.
 If $\langle x \rangle$ does not act irreducibly on $N/Z$, then there exists $Z < N_1<N$ which is $\langle x \rangle$-invariant.  If $N_1$ is cyclic, then, as $\langle x \rangle$ centralizes $\Omega_1(N_1)=Z$, $\langle x \rangle$ centralizes $N_1>Z$, a contradiction.  If $N_1$ is elementary abelian,  then, as $\langle x \rangle$ centralizes $Z$, $[N_1,\langle x \rangle]$ has order at most $p$   by Maschke's Theorem.  If $[N_1,\langle x \rangle]\ne 1$, then $[N_1,\langle x \rangle]\langle x \rangle$ is non-abelian  and $Z$ centralizes $[N_1,\langle x \rangle]\langle x \rangle$, a contradiction.   Hence $\langle x \rangle$ centralizes $N_1$ contrary to $C_N(\langle x \rangle) = Z$.  We conclude that every element of $K_1$ acts irreducibly on $N/Z(N)$. In particular, since $K_1$ is isomorphic to a  subgroup of $\SL_2(p)$, we have that $K_1$ is cyclic of odd order dividing $p+1$. Furthermore, as $K_1$ acts irreducibly on $N/Z(N)$, $N$ has exponent $p$.

 By the definition of $K_1$,  $|K/K_1|$ divides $|\Aut(Z)|= p-1$. Assume that $K \ne K_1$ and let  $y \in K\setminus K_1$ have prime order $r$. Then $r$ does not divide $|K_1|$ and $Z\langle y \rangle$ is non-abelian.  Since $K_1$ centralizes $Z$, we have $C_{K_1}(y) =1$.  Let $w \in K_1$ have prime order $q$. Then $\langle y\rangle\langle w\rangle$ is non-abelian and acts faithfully on $V= N/Z$. Therefore  \cite[27.18]{Asc93} implies that $C_N(y) \ne 1$. As $C_N(y) \cap Z=1$ and $C_N(y)$ centralizes $Z\langle y \rangle$, we have a contradiction. Hence $K=K_1$. Finally, we note that $NK/Z(N)$ is a Frobenius group.

It remains to show that the groups listed are $\X$-groups. We consider the groups listed in (ii) and leave the groups in (i) to the reader.  Assume that $H \le G$ is  non-cyclic. We shall show that $C_G(H) \le H$.  If $H \ge N$, then $C_G(H) \le C_G(N) \le N\le H$ and we are done.   Suppose that $H < N$. Then, as $N$  is extraspecial of exponent $p$, $H$ is elementary abelian of order $p^2$ and $C_N(H)=H$. Since $G/N$ is cyclic of odd order dividing $p+1$, we see that $N_G(H)=N$ and so $C_G(H)=C_N(H)=H$ and we are done in this case. Suppose that $H \not \le N$ and $N \not\le H$. Let $h \in H \setminus N$. Then, as $|G/N|$ divides $p+1$ and is odd, we either have $H \cap N= N$ or $H \cap N= Z$. So we must have $H \cap N= Z= Z(G)$. Now $H/Z \cong G/N$ is cyclic of order dividing $p+1$ and so we get that $H$ is cyclic, a contradiction. Thus $G$ is an $\X$-group.
 \end{proof}

\begin{lemma}\label{abelian} Suppose that $N=F^*(G)$ is elementary abelian of order $p^2$.  Then one of the following holds:
\begin{enumerate}
\item $p=2$,  $G \cong \Sym(4)$ or $\Alt(4)$; or \item $p$ is odd and
$G= NG_0$ is a Frobenius group with Frobenius kernel $N$ and Frobenius complement $G_0$ which is itself  an $\X$-group.  Furthermore, either
\begin{enumerate}\item $G_0$ is cyclic of order dividing $p^2-1$ but not dividing $p-1$;
\item $G_0$ is quaternion;
\item $G_0$ is supersoluble as in part $(2.1.2)$ of Theorem~$\ref{one}$ with $|C|$ dividing $p-\epsilon $ where $p\equiv \epsilon \pmod 4$;
\item $G_0$ is supersoluble as in part $(2.1.3)$ of Theorem~$\ref{one}$ with $D$ a $2$-group, $C_D(C)$ a non-trivial maximal subgroup of $D$ and $ |C|$ odd dividing $p-1$ or $p+1$;
    \item $G_0 \cong \SL_2(3)$;
    \item  $\SL_2(3)\udot 2$ and  $p \equiv \pm 1 \pmod 8$; or
\item $G_0 \cong \SL_2(5) $ and  $60$ divides $p^2-1$.
\end{enumerate}
\end{enumerate}
Furthermore, all the groups listed are $\X$-groups.
\end{lemma}

\begin{proof}
We have $N$ has order $p^2$, is elementary abelian and $G/N$ is isomorphic to a subgroup of $\GL_2(p)$. If $p=2$, then we quickly obtain part (i). So assume that $p$ is odd.

Suppose  that $p$ divides the order of $G/N$. Let $P \in \Syl_p(G)$. Then $P$ is extraspecial of order $p^3$ and $P$ is not normal in $G$.  Hence  by \cite[Theorem 2.8.4]{Gor68} there exists $g \in G$ such that $G \ge K=\langle P, P^g\rangle \cong p^2{:}\SL_2(p)$.  Let $Z= Z(P)$, $t$ be an involution in $K$, $K_0 = C_K(t)$ and $P_0 =  P \cap K_0$.  Then,  as $t$ inverts $N$,  $K_0 \cong \SL_2(p)$,  $P_0$ has order $p$ and centralizes $Z\langle t\rangle$, which is a contradiction as $Z\langle t\rangle \cong \Dih(2p)$. Hence $G/N$ is a $p'$-group.

Suppose that $x \in G\setminus N$. If $C_N(x)\ne 1$, then $C_N(x)$ centralizes $[N,x]\langle x \rangle$ which is non-abelian, a contradiction. Thus $C_N(x) =1$ for all $x \in G \setminus N$.  It follows that $G$ is a Frobenius group with Frobenius kernel $N$. Let $G_0$ be a Frobenius  complement to $N$. As $G_0 \le G$, $G_0$ is an $\X$-group.   Recall that the Sylow $2$-subgroups of $G_0$ are either cyclic or quaternion and that the odd order Sylow subgroups of $G_0$ are all cyclic \cite[V.8.7]{Hu67}.

 Assume that $N$ is not a minimal normal subgroup of $G$.  Then $G/N$ is  conjugate in $\GL_2(p)$ to a subgroup of the diagonal  subgroup. Therefore  $G$ is supersoluble, which is a contradiction. Hence $N$ is a minimal normal subgroup of $G$ and $G_0$ is isomorphic to an irreducible subgroup of $\GL_2(p)$.  This completes the general description of the structure of  $G$. It remains to determine the structure of $G_0$.

  If $G_0$ is nilpotent, then Theorem~\ref{one} (1)  applies to give $G_0$ is either quaternion or cyclic.   In the latter case, as $G_0$ acts irreducibly on $N$ it is isomorphic to a subgroup of the multiplicative group of $\GF(p^2)$ and is not of order dividing $p-1$. This gives the structures in (ii) (a) and (b).

  If $G_0$ is supersoluble, then the structure of $G_0$ is described in part $(2.1)$ of Theorem~\ref{one}, as $\GL_2(p)$ contains no extraspecial subgroups of odd order. We adopt the notation from $(2.1)$.  By \cite[V.8.18 c)]{Hu67}, $Z(G_0) \ne 1$. Hence $(2.1.1)$ cannot occur.  Case  $(2.1.2)$ can occur  and, as $C$ commutes with a non-central cyclic subgroup of order at least $4$ and $G_0$ is isomorphic to a subgroup of $ \GL_2(p)$,  $|C|$ divides $p-1$ if $p \equiv 1 \pmod 4$ and $|C| $ divides $p+1$ if $p \equiv 3 \pmod 4$.  In the situation described in part $(2.1.3)$ of Theorem~\ref{one}, the groups have no $2$-dimensional faithful representations unless  $q=2$ and $C_{D}(C)$ has index $2$. In this case $|C|$ is an odd divisor of $p -1$ or $p+1$.

   Suppose that $G_0$ is not supersoluble.  Refereing to Lemma~\ref{basic1} and using the fact that the Sylow subgroups of $G_0$ are either cyclic or  quaternion, we have that $F^*(G_0)$ is either quaternion of order $8$ or  $F^*(G_0)$ is quasisimple. In the first case we obtain the structures described in parts (b), (e) and (f) from Lemma~\ref{es} where for part (f) we note that we require $\SL_2(p)$ to have order divisible by $16$.

  If $F^*(G_0)$ is quasisimple, then Zassenhaus's Theorem \cite[Theorem 18.6, p. 204]{Passman68} gives $G_0 = WM$ where $W \cong \SL_2(5)$ and $M$ is metacyclic. Since $G_0$ is an $\X$-group, this means that $M \le W$ and $G_0 \cong \SL_2(5)$. Since $\SL_2(5)$ is isomorphic to a subgroup of $\GL_2(p)$ only when $p = 5$ or $60$ divides $p^2-1$ and $p\ne 5$  part (g) holds.

  That $\Sym(4)$ and $\Alt(4)$ are $\X$-groups is easy to check.  The groups listed in (ii) are $\X$-groups by Lemma~\ref{l:two}.
\end{proof}

The finite simple $\X$-groups are determined in \cite{Del13}. We have to extend the arguments to the cases where $F^*(G)$ is simple or quasisimple. This is relatively elementary.

\begin{lemma}\label{s} Suppose that $F^*(G)$ is simple. Then $G \cong \SL_2(4)$, $\PSL_2(9)$, $\Mat(10)$ or $\PSL_2(p)$ where $p$ is a Fermat or Mersenne prime.
\end{lemma}

\begin{proof}
Set $H= F^*(G)$.  As $\X$ is subgroup closed, $H$  is an $\X$-group  and so $H$ is one of the groups listed in the statement by Theorem~3.7 of \cite{Del13}.  Hence we obtain $H \cong \SL_2(4)$, $\PSL_2(9)$ or $\PSL_2(p)$  for $p$  a Fermat or Mersenne prime.

 Suppose that $G>H$. If $H \cong \SL_2(4)$, then $G \cong \Sym(5)$ and the subgroup $2\times \Sym(3)$ witnesses the fact that $\Sym(5)$ is not an $\X$-group.
Suppose $H \cong \PSL_2(9)\cong \Alt(6)$. If $G \ge K \cong \Sym(6)$, then $G$ contains $\Sym(5)$ which is impossible. Therefore $G \cong \PGL_2(9)$ or $G \cong \Mat(10)$.
In the first case, $G $ contains a subgroup $\Dih(20) \cong 2 \times \Dih(10)$ which is impossible. Thus $G \cong \Mat(10)$ and this group is easily shown to satisfy the hypothesis as all the centralizer of elements of prime order are $\X$-groups.

If $H \cong    \PSL_2(p)$,  $p$ a Fermat or Mersenne prime, then $G \cong \PGL_2(p)$ and contains a dihedral group of order $2(p + 1)$ and one of order $2(p - 1)$. One of these is not a $2$-group and this contradicts $G$ being an $\X$-group.
\end{proof}

\begin{lemma}\label{qs}
Suppose that $F^*(G)$ is quasisimple but not simple. Then $F^*(G) \cong \SL_2(p)$ where $p$ is a Fermat prime, $|G/H|\le 2$ and $G$ has quaternion Sylow $2$-subgroups.
\end{lemma}

\begin{proof}
Let $H= F^*(G)$ and $Z= Z(H)$. Since $H$ centralizes $Z$, we have $Z$ is cyclic. Let $S \in\Syl_2(H)$. If $Z \not \le S$, then $S$ must be cyclic. Since groups with a cyclic Sylow $2$-subgroup have a normal $2$-complement \cite[39.2]{Asc93}, this is impossible. Hence $Z \le S$. In particular, $Z(G) \ne 1$ as the central involution of $H$ is central in $G$. It follows also that all the odd order Sylow subgroups of $G$ are cyclic.
By part (1) of Theorem~\ref{one}, $S$ is either abelian, dihedral, semidihedral or quaternion.  If $S$ is abelian, then $S/Z$ is cyclic and again we have a contradiction. So $S$ is non-abelian. Thus $S/Z$ is dihedral (including elementary abelian of order $4$).  Hence $H/Z \cong \Alt(7)$ or $\PSL_2(q)$ for some odd prime power $q$ \cite[Theorem 16.3]{Gor68}. Since  the odd order Sylow subgroups of $G$ are cyclic,  we deduce that  $H \cong \SL_2(p)$ for some odd prime $p$. If $p-1$ is not a power of $2$, then $H$ has a non-abelian subgroup of order $pr$  where $r$ is an odd prime divisor of $p-1$  which   is centralized by $Z$. Hence $p $ is a Fermat prime.

Suppose that $G>H$ with $H \cong \SL_2(p)$, $p$ a Fermat prime. Note $G/H$ has order $2$.  Let $S \in \Syl_2(G)$.  Then $S \cap H$ is a quaternion group.
Suppose that $S$ is not quaternion  Then there is an involution  $t \in S \setminus H$. By the Baer-Suzuki Theorem, there exists a dihedral group $D$ of order $2r$ for some odd prime $r$ which contains $t$. Since $D$ and $Z$ commute, this is impossible. Hence $S$ is quaternion. This gives the structure described in the lemma.

It remains to demonstrate that the groups $\SL_2(p)$ and $\SL_2(p)\udot 2$ with $p$ a Fermat prime are indeed $\X$-groups. Let $G$ denote one of these group, $H =F^*(G)\cong \SL_2(p)$. Recall from the comments just after the statement of Theorem~\ref{one} that $G$ is isomorphic to a subgroup of $X=\SL_2(p^2)$. Let $V$ be the natural $\GF(p^2)$ representation of $X$ and thereby a representation of $G$.  Assume that $L \le G$ is non-cyclic. Since $H$ has no abelian  subgroups which are not cyclic, $L$ is non-abelian and $L$ acts irreducibly on $V$. Schur's Lemma implies that $C_{X}(L)$ consists of scalar matrices and so has order at most $2$. If $L$ has even order, then as $G$ has quaternion Sylow $2$-subgroups, $L \ge C_G(L)$.  So suppose that $L$ has odd order.  Then using Dickson's Theorem \cite[260, page 285]{Dickson58}, as $p$ is a Fermat prime, we find that $L$ is cyclic, a contradiction.  Thus $G$ is an $\X$-group.
\end{proof}

\begin{proof}[Proof of Theorem~$\ref{one}$]  This follows from the combination of the lemmas in this section.
\end{proof}

\section{Infinite locally finite $\X$-groups}
\label{s:locallyfinite}

It has been proved in \cite[Theorem 2.2]{Del13} that an infinite abelian group is in the class $\X$ if and only if it is either cyclic or isomorphic to $\Z_{p^\infty}$  (the Pr\"ufer $p$-group) for some prime $p$. Moreover, Theorem 2.3 and Theorem 2.5 of \cite{Del13} imply that  every infinite nilpotent $\X$-group is abelian. We start this section by showing that some extensions of infinite abelian $\X$-groups provide further examples of infinite $\X$-groups.

\begin{lemma}
\label{l:dihedral}
The infinite dihedral group belongs to the class $\X$.
\end{lemma}
\begin{proof}  Write $G=\langle a,y\,|\,y^2=1, a^y=a^{-1}\rangle$. Then for every non-cyclic subgroup $H$ of $G$ there exist non-zero integers $n$ and $m$ such that $a^n, a^my\in H$. It easily follows that $C_G(H)= 1 $.
\end{proof}

\begin{lemma}
\label{l:2adic}
Let $G=A\langle y\rangle$ where $A\cong\Z_{2^\infty}$ and $\langle y\rangle$ has order $2$ or $4$, with $y^2\in A$ and $a^y=a^{-1}$, for all $a\in A$. Then $G$ belongs to the class $\X$.
\end{lemma}
\begin{proof}
It is clear that $G/A$ has order $2$, and $A$ is the Fitting subgroup of $G$. Also $C_G(A)=A$  and $Z(G)$ is the subgroup of order $2$ of $A$. Let   $H$ be a non-cyclic subgroup of $G$ with $H\ne A$. Then $H \not \le A$ as every proper subgroup of $A$ is cyclic.  Pick any element $h\in H\setminus A$.
Then $G=A\langle h\rangle$ since $|G:A|=2$. Therefore by the Dedekind modular law we get $H=C\langle h\rangle$, where $C=A\cap H>1$ is finite.

Since $h=bv$ with $b\in A$ and $v\in\langle y\rangle \setminus A$, we get $a^h=a^{-1}$ for all $a\in A$.  In particular,  $C_A(h)$ has order $2$ and $C_G(h)$ has order $4$.  Since $C$ has a unique involution and $h \in C_G(H)$, we conclude that $C_G(H) \le H$ and so $G$ is an $\X$-group.
\end{proof}
When $\langle y\rangle$ has order $2$,   the group $G=A\rtimes\langle y\rangle$ of Lemma~\ref{l:2adic} is a generalized dihedral group.

Let   $p$ denote any odd prime. Then, by Hensel's Theorem (see for instance \cite[Theorem 127.5]{Fuc73}), the group $\Z_{p^\infty}$ has an automorphism of order $p-1$, say $\phi$.

\begin{lemma}
\label{l:padic}
The groups $G=\Z_{p^\infty}\rtimes\langle \phi^j\rangle$ for $1 \le j\le p-1$   are  $\X$-groups.
\end{lemma}
\begin{proof} As $\X$ is subgroup closed, it suffices to show that $G=\Z_{p^\infty}\rtimes\langle \phi\rangle$ is an $\X$-group.
Write the elements of $G$ in the form $ay$ with $a\in A\cong\Z_{p^\infty}$ and $y\in\langle\phi\rangle$. Suppose there exist non-trivial elements $a\in A$ and $y\in \langle\phi\rangle$ such that $a^y=a$. For a suitable non-negative integer $n$, the element $a^{p^n}$ has order $p$ and it is fixed by $y$. Then $y$ centralizes  all elements of order $p$ in  $A$, and therefore $y=1$ by a result due to Baer (see, for instance, \cite[Lemma 3.28]{Rob72}). This contradiction shows that $\langle\phi\rangle$ acts fixed point freely on $A$.

Let $H$ be any non-cyclic subgroup of $G$. Then, as $G/A$ is cyclic, $A\cap H\ne 1$.  If $H=A$ then of course $C_G(H)=H$. Thus we can assume that there exist non-trivial elements $a,b\in A$ and  $y\in \langle\phi\rangle$ such that $a,by\in H$. Let $g\in C_G(H)$. If $g\in A$ then $1=[g,by]=[g,y]$, so $g=1$. Now let $g=cz$ with $c\in A$ and $1\ne z\in\langle\phi\rangle$. Thus $1=[cz,a]=[z,a]$, and $a=1$, a contradiction. Therefore $C_G(H)\leq H$ for all non-cyclic subgroups $H$ of $G$, so $G$ is an $\X$-group.
\end{proof}

\begin{lemma}
\label{t:polycyclic}
An infinite polycyclic group belongs to the class $\X$ if and only if it is either cyclic or dihedral.
\end{lemma}
\begin{proof} Arguing as in the proof of Theorem 3.1 of \cite{Del13}, one can easily prove that every infinite polycyclic $\X$-group is either cyclic or dihedral. On the other hand, the infinite dihedral group belongs to the class $\X$ by Lemma~\ref{l:dihedral}.
\end{proof}

\begin{proposition}
\label{t:torsionfreesoluble}
A torsion-free soluble group belongs to the class $\X$ if and only if it is cyclic.
\end{proposition}
\begin{proof} Let $G$ be a torsion-free soluble $\X$-group. Then every abelian subgroup of $G$ is cyclic, so $G$ satisfies the maximal condition on subgroups by a result due to Mal'cev (see, for instance, \cite[15.2.1]{Rob96}). Thus $G$ is polycyclic by \cite[5.4.12]{Rob96}. Therefore $G$ has to be cyclic.
\end{proof}

In next theorem we determine all infinite soluble $\X$-groups.

\begin{theorem}
\label{t:soluble}
Let $G$ be an infinite soluble group. Then $G$ is an $\X$-group if and only if one of the following holds:
\begin{enumerate}
\item $G$ is cyclic;
\item $G\cong\Z_{p^\infty}$ for some prime $p$;
\item $G$ is dihedral;
\item $G=A\langle y\rangle$ where $A\cong\Z_{2^\infty}$ and $\langle y\rangle$ has order $2$ or $4$, with $y^2\in A$ and $a^y=a^{-1}$, for all $a\in A$;
\item $G\cong A\rtimes D$, where $A\cong\Z_{p^\infty}$ and $ 1 \neq D\leq C_{p-1}$ for some odd prime $p$.
\end{enumerate}
\end{theorem}
\begin{proof} First let $G$ be an $\X$-group. If $G$ is abelian then (i) or (ii) holds by  \cite[Theorem 2.2]{Del13}. Assume $G$ is non-abelian, and let $A$ be the Fitting subgroup of $G$. Then $A\not=1$ and $C_G(A)\leq A$ as $G$ is soluble. Let $N$ be a nilpotent normal subgroup of $G$.
 Then $N$ is finite, as, otherwise, using $N$ is self-centralizing and $G/Z(N)$ is a
subgroup of $\Aut(N)$, we obtain $G$ is finite, which is a contradiction. Thus  \cite[Theorems 2.3
and 2.5]{Del13}  imply that $N$ is abelian. In particular, as the product of any two normal nilpotent subgroups of $G$ is again a normal nilpotent subgroup by Fitting's Theorem, we see that the generators of $A$ commute. Hence $A$ is abelian.   As $A$ is infinite and abelian, $A=C_G(A)$ is either infinite cyclic or isomorphic to $\Z_{p^\infty}$ for some prime $p$.  In the former case clearly $G'\leq A$. In the latter case, let $C$ be any proper subgroup of $A$. Thus $C$ is finite cyclic. Moreover $C$ is characteristic in $A$, so it is normal in $G$, and $G/C_G(C)$ is abelian since it is isomorphic to a subgroup of $\Aut(C)$. It follows that $G'\leq C_G(C)$, and again $G'\leq C_G(A)=A$. Therefore $G/A$ is abelian.

If $A$ is infinite cyclic, then the argument in the proof of Theorem 3.1 of \cite{Del13} shows that $G$ is dihedral. Thus (iii) holds.

Let  $A\cong\Z_{p^\infty}$ for some prime $p$, and suppose there exists an element $x\in G$ of infinite order. Then  $x\in G\setminus A$,  and so there exists an element $y\in A$ such that $[x,y]\not=1$.
Then $\langle y\rangle$ is a finite normal subgroup of $G$, so conjugation by $x$ induces  a non-trivial automorphism of $\langle y\rangle$. Since $\Aut(\langle y\rangle)$ is finite, it follows that there is an integer $n$ such that $[x^n,y]=1$. Now $y$ is a torsion element and $x^n$ has infinite order and so $\langle x^n,y\rangle$ is neither periodic nor torsion-free and this  contradicts   \cite[Theorems 2.2]{Del13}.  Therefore $G$ is periodic, and $G/A$ is isomorphic to a periodic subgroup of automorphisms of $\Z_{p^\infty}$.

It is well-known that $\Aut(\Z_{p^\infty})$ is isomorphic to the multiplicative group of all $p$-adic units. It follows that periodic automorphisms of $\Z_{p^\infty}$ form a cyclic group having order $2$ if $p=2$, and order $p-1$ if $p$ is odd (see, for instance, \cite{Has49} for details). In the latter case (v) holds. Finally, let $p=2$. Then $G/A=\langle yA\rangle$ has order $2$, and $G=A\langle y\rangle$ with $y\notin A$ and $y^2\in A$. Moreover $a^y=a^{-1}$, for all $a\in A$. If $y$ has order $2$ then $G=A\rtimes\langle y\rangle$. Otherwise from $y^2\in A$ it follows $y^2=(y^2)^y=y^{-2}$, so $y$ has order $4$. Therefore $G$ has the structure described in (iv).

On the other hand, Lemmas~\ref{l:dihedral} -- \ref{l:padic} show that the groups listed in (i) -- (v)  are  $\X$-groups. 
\end{proof}

Finally, we determine all infinite locally finite $\X$-groups.

\begin{theorem}
\label{t:locallyfinite}
Let $G$ be an infinite locally finite group. Then $G$ is an $\X$-group if and only if one of the following holds:
\begin{enumerate}
\item $G\cong\Z_{p^\infty}$ for some prime $p$;
\item $G=A\langle y\rangle$ where $A\cong\Z_{2^\infty}$ and $\langle y\rangle$ has order $2$ or $4$, with $y^2\in A$ and $a^y=a^{-1}$, for all $a\in A$;
\item $G\cong A\rtimes D$, where $A\cong\Z_{p^\infty}$ and $1\neq D\leq C_{p-1}$ for some odd prime $p$.
\end{enumerate}
\end{theorem}

\begin{proof} Any abelian subgroup of $G$ is either finite or isomorphic to $\Z_{p^\infty}$ for some prime $p$, so it satisfies the minimal condition on subgroups. Thus $G$ is a \v Cernikov group by a result due to \v Sunkov (see, for instance  \cite[page 436, I]{Rob96}). Hence there exists an abelian normal subgroup $A$ of $G$ such that $A\cong \Z_{p^\infty}$ for some prime $p$, and $G/A$ is finite. It follows that $G$ is metabelian, arguing as in the proof of Theorem~\ref{t:soluble}. Therefore the result follows from Theorem~\ref{t:soluble}.
\end{proof}

Clearly Theorem~\ref{one} and Theorem~\ref{t:locallyfinite} give a complete classification of locally finite $\X$-groups.
\begin{corollary}
\label{c:locallynilpotent}
Let $G$ be an infinite locally nilpotent group. Then $G$ is an $\X$-group if and only if one of the following holds:
\begin{enumerate}
\item $G$ is cyclic;
\item $G\cong\Z_{p^\infty}$ for some prime $p$;
\item $G=A\langle y\rangle$ where $A\cong\Z_{2^\infty}$ and $\langle y\rangle$ has order $2$ or $4$, with $y^2\in A$ and $a^y=a^{-1}$, for all $a\in A$.
\end{enumerate}
\end{corollary}

\begin{proof}
Suppose $G$ is not abelian. Every finitely generated subgroup of $G$ is nilpotent, so it is either abelian or finite. It easily follows that all torsion-free elements of $G$ are central. Thus $G$ is periodic (see \cite[Proposition 1]{Del07}). Therefore $G$ is direct product of $p$-groups (see, for instance, \cite[Proposition~12.1.1]{Rob96}). Actually only one prime can occur since $G$ is an $\X$-group, so $G$ is a locally finite $p$-group. Thus the result follows by Theorem~\ref{t:locallyfinite}.
\end{proof}

\section*{References}

\bibliography{mybibfile}

\end{document}